\documentclass[12pt]{amsart}
\usepackage{amsmath,amssymb,amsthm,mathtools}
\usepackage{fourier}

\theoremstyle{plain}
\newtheorem{theorem}{Theorem}
\newtheorem{lemma}[theorem]{Lemma}
\newtheorem*{theorem*}{Theorem}

\newcommand{\N}{\mathbb{N}}
\newcommand{\Z}{\mathbb{Z}}

\begin{document}
\author{Iwan Praton}
\author{Weiran Zeng}
\address{Franklin \& Marshall College} \email{ipraton@fandm.edu, wzeng@fandm.edu}
\title{Amicable Triangles and Rectangles on the Integer Lattice}
\date{}
\begin{abstract}
Two polygons are amicable if the perimeter of one is equal to the area of the other and vice versa. A polygon is a lattice polygon if its vertices are on the integer lattice $\Z^2$. We show that there is one pair of amicable lattice triangles and five pairs of amicable lattice rectangles.
\end{abstract}

\maketitle

A planar polygon is \emph{equable} if its area is equal to its perimeter. Note that this definition is not scale invariant: a polygon that is twice as big as an equable polygon is not itself equable. So we cannot find infinitely many equable polygons simply by scaling up or down. Unfortunately, without any additional restrictions, equable polygons are still rather easy to find. So we concentrate here on equable polygons whose vertices lie on the integer lattice $\Z^2$; such polygons are called \emph{lattice} polygons. In this case it has long been known that there are only five equable triangles; see the appendix of \cite{AC1} for a proof. The authors of \cite{AC1} have systematically investigated equable polygons on the integer and Eisenstein lattices; besides \cite{AC1}, see \cite{AC2}, \cite{AC3}, and others in the series. 

We wish to take a different path here. Two polygons are \emph{amicable} if the area of the first is equal to the perimeter of the second, and vice versa. Note that if a polygon is equable, then it is amicable with itself. To avoid this egotistical possibility, we henceforth assume that the two amicable polygons are not the same. 

Geometric amicability was introduced in \cite{PS}; there it is shown that there is only one pair of amicable heronian triangles (i.e., triangles with integer sides and integer areas). Here we investigate amicable triangles and rectangles on the integer lattice. We show that there is only one pair of amicable lattice triangles (the same as the pair in \cite{PS}).  In a bit of numerological coincidence, we also show that there are five pairs of amicable lattice rectangles, mirroring the five equable lattice triangles.

\section*{Triangles}
In this section we prove the following theorem.
\begin{theorem}
There is exactly one pair of amicable lattice triangles.
\end{theorem}
Of course, ``only one pair'' means up to rigid motions. From \cite{PS} we know that there is only one pair of amicable heronian triangles. The pair consists of a triangle with side lengths 3, 25, 26 and a companion triangle with side lengths 9, 12, 15. These are lattice triangles: we can see this from \cite{Y}, or more explicitly, by noticing that the first triangle can be positioned so that its vertices are at $(0,0)$, $(0,9)$, $(12,0)$, and the second triangle can similarly be located so that its vertices are at $(0,0)$, $(24,7)$, $(24,10)$. To complete the proof, we show that amicable lattice triangles must be heronian. We start with a general about amicable polygons.
\begin{lemma}\label{lemma1}
An amicable polygon has integer side lengths.
\end{lemma}
\begin{proof}
This is basically the same as Remark 2 in \cite{AC1}. We provide a recap here for completeness. Consider first a lattice triangle; we can assume that its vertices are at $(0,0)$, $(p,q)$, and $(r,s)$. Its area is $\frac12|ps-qr|$, so it is rational. Now a lattice polygon can be partitioned into lattice triangles, so the area of a lattice polygon is also rational.

If a lattice polygon is part of an amicable pair, then its parameter is equal to the area of its companion lattice polygon, hence its parameter is also rational. The side length of a lattice polygon is the distance of two points in $\Z^2$, so it is of the form $\sqrt{a}$ where $a$ is an integer. Therefore the perimeter of an amicable lattice polygon is of the form $\sum_{i=1}^n \sqrt{a_i}$, where $a_i\in \N$. This sum is rational. It is well-known that if $\sum_{i=1}^n \sqrt{a_i}$ is rational and $a_i\in \N$, then each $\sqrt{a_i}$ is itself rational, hence an integer. (There is a short proof of this in \cite{A}, but it uses Cogalois theory.) This finishes the proof. 

We actually only need this result for triangles and rectangles, so we provide elementary proofs for these cases that do not depend on field extensions. First, suppose that $\sqrt{x}+\sqrt{y}$ is rational, where $x,y\in\N$. Then $\sqrt{x}-\sqrt{y} = (x-y)/(\sqrt{x}+\sqrt{y})$ is also rational. Therefore $\sqrt{x}=((\sqrt{x}+\sqrt{y})-(\sqrt{x}-\sqrt{y}))/2$ is also rational, hence an integer. Similarly, $\sqrt{y}$ is an integer.

Now suppose $d=\sqrt{a}+\sqrt{b}+\sqrt{c}$ is rational. Then $(\sqrt{a}+\sqrt{b})^2 = (d-\sqrt{c})^2$, which implies that $a+b+2\sqrt{ab}=d^2+c-2d\sqrt{c}$, i.e., $\sqrt{ab}+d\sqrt{c} = (d^2+c-a-b)/2$, which is rational. By the previous result about the sum of two roots, $d\sqrt{c}$ is an integer, so $\sqrt{c}$ is rational. Therefore $\sqrt{a}+\sqrt{b}=d-\sqrt{c}$ is also rational, and the result about the sum of two roots again allows us to conclude that $\sqrt{a}$ and $\sqrt{b}$ are integers.
\end{proof}

We can now show that amicable lattice triangle are heronian. By Lemma~\ref{lemma1} the side lengths of the triangles are integers, so we only need to show that their areas are also integers. This follows immediately from amicability: the areas are equal to the perimeters, which are integers.

\section*{Rectangles}

Let $(R_1,R_2)$ be a pair of amicable lattice rectangles. By Lemma~\ref{lemma1}, the side lengths of both rectangles are integers. Let $a$ and $b$ denote the side lengths of $R_1$, with $a\leq b$, and let $x$ and $y$ denote the side lengths of $R_2$, with $x\leq y$. Because $R_1$ and $R_2$ are amicable, one of these rectangles must have its perimeter be larger than (or equal to) its area. We choose $R_1$ to be that rectangle and we focus our attention on it. There are not many such rectangles, since integer areas are usually larger than perimeters.

\begin{lemma}
Let $R_1$ be  as above,  so its side lengths are $a$ and $b$ where $a\leq b$, and its area is at most equal to its perimeter, i.e., $ab\leq 2(a+b)$. Then $a=1$ or $a=2$.
\end{lemma}
\begin{proof}
If $a> 4$, then $ab> 4b =2b+2b\geq 2a+2b$, contradicting $ab \leq 2(a+b)$. Thus $a\leq 4$. If $a=4$, then $ab \leq  2(a+b)$ becomes $b\leq 4$, which implies that $b=4$ (since $b\geq a$). Thus $R_1$ is a $4$-by-$4$ square; its area and perimeter are both 16. Its only amicable partner is itself, so we rule it out.

We now consider the possibility that $a=3$. Then $ab\leq 2(a+b)$ becomes $b\leq 6$, so $b=3, 4, 5$, or $6$. If $b=3$ or $b=5$, then the area of $R_1$ is odd, so it cannot be equal to the perimeter of $R_2$, which is even. If $b=4$, then the area of $R_1$ is 12 and the perimeter is 14. So we must have $2(x+y)=12$ and $xy=14$. It is to check that there are no integer solutions to these two equations. We can also check that  $b=6$ leads to $R_1$ being the same rectangle as $R_2$, which we already ruled out. Therefore $a=1$ or $a=2$.
\end{proof}

We can now prove our main theorem.

\begin{theorem}
The following pairs of lattice rectangles are amicable:
\begin{itemize}
\item
1-by-34 and 7-by-10
\item 1-by-38 and 6-by-13
\item 1-by-54 and 5-by-22
\item 2-by-10 and 4-by-6
\item 2-by-13 and 3-by-10
\end{itemize}
These are the only amicable lattice rectangles.
\end{theorem}
\begin{proof}
It is straightforward to check that these five pairs are amicable. We need to show that there are no others.

Suppose the amicable pair, as above, consists of the $a\times b$ and the $x\times y$ rectangles, where $ab\leq 2(a+b)$, $a\leq b$, and $x\leq y$. 
By amicability we have $ab=2(x+y)$ and $2(a+b)=xy$. We consider these equations as a linear system of equations in two unknowns $b$ and $y$:
\[
\begin{pmatrix}
a & -2\\
-2 & x\\
\end{pmatrix}
\begin{pmatrix}
b \\
y\\
\end{pmatrix}
=
\begin{pmatrix}
2x \\
2a\\
\end{pmatrix}
\]
If $ax\neq 4$, then there is a unique solution: 
\[
b=\frac{2x^2+4a}{ax-4}; \quad y=\frac{4x+4a}{ax-4}.
\]
Thus $ax-4$ must be a divisor of $2x^2+4a$ and $4x+4a$. But by the previous lemma, either $a=1$ or $a=2$, so this divisibility requirement can be treated separately.

First suppose $a=1$. If $x=4$, then the amicability equations have no solution. If $x\neq 4$, then $b=(2x^2+4)/(x-4)$ and $y=(4x+2)/(x-4)$. Let $c=x-4$. The second equation implies that $4x+2=4(c+4)+2=4c+18$ is divisible by $c$. Thus $c$ must be a divisor of $18$, so the possible values of $c$ are $1, 2, 3, 6, 9, 18$. The corresponding values of $x$, $y$, and $b$ are $5,6,7,10,13,22$; $22,13,10,7,6,5$; and $54,38,34,34,38,54$. These values lead to the first three entries in the theorem.

Now suppose $a=2$. If $x=2$, the amicability equations have no solution. If $x\neq 2$, then $b=(x^2+4)/(x-2)$ and $y=(2x+4)/(x-2)$. Arguing as before, we see that $x-2$ must be a divisor of $8$, and hence the possible values of $x$ are $3,4,6,10$. The corresponding values of $y$ and $b$  are, in order, $10,6,4,3$ and $13,10,10,13$. These lead to the last two entries in the theorem.
\end{proof}

\end{document}